\documentclass[11pt]{amsart}
\usepackage[dvips]{graphicx}
\usepackage[margin=1in]{geometry}

\usepackage[dvips]{graphicx}

\usepackage{graphicx}
\usepackage{amssymb}
\usepackage{amsmath}
\usepackage{color}
\usepackage{times}

\usepackage[unicode,bookmarks,colorlinks]{hyperref}
\hypersetup{
    linkcolor=brickred,
}

\definecolor{mahogany}{cmyk}{0, 0.77, 0.87, 0}
\definecolor{salmon}{cmyk}{0, 0.53, 0.38, 0}
\definecolor{melon}{cmyk}{0, 0.46, 0.50, 0}
\definecolor{yellowgreen}{cmyk}{0.44, 0, 0.74, 0}
\definecolor{brickred}{cmyk}{0, 0.89, 0.94, 0.28}
\definecolor{OliveGreen}{cmyk}{0.64, 0, 0.95, 0.40}
\definecolor{RawSienna}{cmyk}{0, 0.72, 1.0, 0.45}
\definecolor{ZurichRed}{rgb}{1, 0, 0} 

\newcommand{\n}{\noindent}

\newcommand{\R}{\mathbb{R}}                  
\newcommand{\E}{\mathbb{E}}

\numberwithin{equation}{section}

\theoremstyle{plain}
\newtheorem{thm}{Theorem}[section]
\newtheorem{lem}[thm]{Lemma}

\theoremstyle{definition}

\newtheorem{remark}{Remark}[section]

\begin{document}

\title[First Eigenfunction of Symmetric stable processes]
{On the First Eigenfunction of the Symmetric Stable Process in a Bounded Lipschitz Domain}

\author{Rodrigo Ba\~nuelos}\thanks{R. Ba\~nuelos is supported in part  by NSF Grant
\#0603701-DMS}
\address{Department of Mathematics, Purdue University, West Lafayette, IN 47907}
\email{banuelos@math.purdue.edu}
\author{Dante DeBlassie}
\address{Department of Mathematical Sciences,
New Mexico State University, 
P. O. Box 30001,
Department 3MB, 
Las Cruces, NM 88003-8001}
\email{deblass@nmsu.edu}

\begin{abstract}
We give a proof that the first eigenfunction of the $\alpha$-symmetric stable process on a bounded Lipschitz domain in $\R^d$, $d\geq 1$, is superharmonic for $\alpha=2/m$, where $m>2$ is an integer.  This result was first proved for the ball by M. Ka{\ss}mann and L. Silvestre (personal communication) with different methods.  For $\alpha=1$,  the result was proved in \cite[Theorem 4.7]{BanKul}.

\end{abstract}
\maketitle
\maketitle

\section{Introduction}\label{sec1}

\indent 
For $\alpha\in(0,2)$ and $d\geq 1$, let $X_t$ be the $d$-dimensional $\alpha$-symmetric stable process.  This is the process with stationary independent increments whose transition density
\[
p(t,x,y)=p(t,x-y),\quad (t,x,y)\in(0,\infty)\times\mathbb{R}^d\times\mathbb{R}^d
\]
is characterized by its Fourier transform
\[
\int_{\mathbb{R}^d}e^{iy\cdot\xi}p(t,y)\,dy=\exp(-t|\xi|^\alpha),\quad t>0, \quad\xi\in\mathbb{R}^d.
\]
When $\alpha=2$, this reduces to $d$-dimensional Brownian motion run at twice its usual speed.

The potential theory for $X_t$, $0<\alpha<2$, has been the subject of intense study for quite a few years and many of the well known results for Brownian motion have been extended for these $\alpha$'s (and even to many other L\'evy processes). More recently, there have been many efforts to extend the detailed and refined spectral theoretic properties of Brownian motion to the general case of $0<\alpha<2$. Substantial progress has been made but many basic questions remain open. For some of this literature,  we refer the reader to \cite{BanMen, BanKul, BanKulSiu, Bog-others, KimSonVon, Kwa}, and the many references given in these papers. The result in this paper arises from  problems first raised in Ba\~nuelos and Kulczycki \cite{BanKul} (see the introduction to that paper) and Ba\~nuelos, Kulczycki and M\'endez-Hern\'andez \cite[Question 1.1, Conjecture 1.2]{BanKulMen},  concerning the shape of the first eigenfunction for the semigroup generated by the stable process killed upon leaving a domain $D$, or equivalently, for the first 
Dirichlet eigenfunction for the fractional Laplacian with Dirichlet boundary  conditions. As discussed in \cite{BanKul}, these problems  were originally motivated  by the classical result of Brascamp and Lieb  \cite{BraLie} which states that for  the Laplacian (the case $\alpha=2$) the eigenfunction  is $\log$--concave when the domain is  convex.   Since one way to obtain this result is to show that the finite dimensional distributions of Brownian motion are log--concave in the starting point when the sets are convex, one would expect such an approach to yield results also for stable processes and even for more general L\'evy processes.  This approach rests on Pr\'ekipa's result that multiple convolutions of log--concave functions are $\log$-concave and it fails for $\alpha\not= 2$ as the transition densities are not $\log$-concave. For more on this approach and what it gives for stable processes (``mid--concavity"), see  \cite{BanKulMen}. By exploiting the connection with a mixed Steklov problem it is proved in \cite[Theorem 4.7]{BanKul} that when $\alpha=1$--the case of the Cauchy processes--the first eigenfunction is superharmonic for any Lipschitz domain $D\subset\R^d$, $d\geq 1$.  As a corollary of this result one obtains that the eigenfunction for the interval $D=(-1,1)$ for the Cauchy process  is in fact concave (hence $\log$-concave) as in the case of the Laplacian.  In \cite{Deb},  DeBlassie made an analogous connection for \emph{rational} values of $\alpha$ and solutions of PDEs involving certain higher order operators. In this paper we make use of this connection to study the situation when $\alpha=2/m$,  where $m>2$ is an integer.

To set the stage, let $D\subseteq\mathbb{R}^d$ be an open set with finite volume. Denote by $\tau_D$ the first exit time of $X_t$ from $D$ and $\E_x$ the expectation associated with $X_0=x$. Then the operator (see \cite{BanKul} and \cite{Dav}) 
\[
P^D_tf(x)=\E_x[f(X_t); \tau_D>t],\quad x\in D,\quad t>0,\quad f\in L^2(D)
\]
generates a self-adjoint ultracontractive semigroup on $L^2(D)$ and hence there is an orthonormal basis of eigenfunctions $\{\varphi_n\}$ in $L^2(D)$ with corresponding eigenvalues $\{\lambda_n\}$ such that
\[
P^D_t\varphi_n=-\lambda_n\varphi_n,\quad\text{on $D$}
\]
and
\[
0<\lambda_1<\lambda_2\leq\lambda_3\leq\cdots,
\]
with $\lambda_n\to\infty$ as $n\to\infty$. Moreover, $\varphi_1>0$ on $D$ and each eigenfunction is bounded and continuous on $D$.  Our result in this paper extends the result in Ba\~nuelos and Kulczycki, \cite[Theorem 4.7]{BanKul} for $\alpha=1$ as well as the Ka{\ss}man-Silvestre result in the ball cited in the abstract.

\begin{thm}\label{thm1.1}
Suppose $D\subseteq\mathbb{R}^d$ is a bounded Lipschitz domain, where $d\geq 1$. Let $\alpha=2/m$, where $m>2$ is an integer. Then
\begin{equation}\label{eq1.1}
\Delta\varphi_1\leq 0\quad\text{on $D$}.
\end{equation}
\end{thm}

\begin{remark} As mentioned above, this result was first proved for the ball by 
M. Ka{\ss}mann and L. Silvestre (personal communication).  Their method is completely different from ours. Our proof of Theorem \ref{thm1.1} rests on various extensions of the results in \cite{BanKul} and \cite{Deb}.  
\end{remark}

\section{Preliminaries}\label{sec2}

In this section we collect some results on stable subordinators. We also summarize and expand upon some of the work in \cite{BanKul} and \cite{Deb} that we use below. The extensions are easily obtained using the methods and ideas in those papers and we omit the details.

Let $P_t$ be the semigroup associated with the $\alpha$-symmetric stable process $X_t$ defined by
\[
P_tf(x)=\int_{\mathbb{R}^d}p(t,x,y)f(y)\,dy,\quad f\in L^1(\mathbb{R}^d).
\]
Define
\begin{equation}\label{eq2.1}
u_n(t,x)=P_t\varphi_n(x),\quad n\geq 1
\end{equation}
(recall $\varphi_n$ is the $n^{\text th}$ eigenfunction associated with $P_t^D$). 
It is well known that $X_t$ can be obtained by subordinating a Brownian motion. More precisely, by running a $d$-dimensional Brownian motion at twice an independent $\alpha/2$-stable subordinator, we obtain the symmetric $\alpha$-stable process $X_t$. The density $f_t(x)$ of the $\alpha/2$-stable subordinator has Laplace transform
\[
\int_0^\infty e^{-\lambda s}f_t(s)\,ds=e^{-t\lambda^{\alpha/2}}.
\]
Thus if
\[
g(s,x,y)=(4\pi s)^{-d/2}\exp(-|x-y|^2/4s),
\]
then the transition density of $X_t$ can be expressed as
\[
p(t,x,y)=\int_0^\infty g(s,x,y)f_t(s)\,ds.
\]

The density $f_t$ has the following properties:

\begin{enumerate} 
\item[ ]Scaling: 
\begin{equation}\label{eq2.2}
f_t(s)=t^{-2/\alpha}f_1(t^{-2/\alpha}s).
\end{equation}
\item[ ] 
For each nonnegative integer $q$, there are $a_j(q)$, $j=0,1,\dots, q$ such that 
\begin{equation}\label{eq2.3}
\frac{\partial^q}{\partial t^q}f_t(s)=\sum_{j=0}^q a_j(q)\,t^{-2/\alpha-q}\left(t^{-2/\alpha}s\right)^jf_1^{(j)}(t^{-2/\alpha}s). 
\end{equation}
\item[ ] If we set 
\[
\gamma_k=\frac{(-1)^{k+1}\Gamma(k\alpha/2+1)}{\pi k!}\sin\left(\frac{\pi k\alpha}{2}\right),
\]
then we have
\begin{equation}\label{eq2.4}
f_1(s)=\sum_{k=1}^\infty \gamma_ks^{-k\alpha/2-1}.
\end{equation}
\item[ ] For any nonnegative integer $n$,
\begin{equation}\label{eq2.5}
f_t^{(n)}(s)\to 0,\,\,\, \text{ as $s\to 0^+$.}
\end{equation}
\end{enumerate} 
The formula \eqref{eq2.3} follows from \eqref{eq2.2}.   The expression \eqref{eq2.4} can  be found in Zolotarev \cite[p.~90, (2.4.8)]{Zol} or Feller \cite[p.~583, Lemma 1]{Fel}. 
The limit in \eqref{eq2.5} is from Zolotarev \cite{Zol}; see Section 2.5, Theorem 2.5.3 and Remark 1.

\bigskip
\begin{lem}\label{lem2.1} If $\alpha=2/m$ for some integer $m>2$, then
\[
\left(\frac{\partial}{\partial s}-(-1)^m\frac{\partial^m}{\partial t^m}\right)f_t(s)=0.
\]
\end{lem}
\begin{proof}
It is not hard to show that for $\lambda>0$ and $q>0$,
\[
\frac{\partial^q}{\partial t^q}\int_0^\infty e^{-\lambda s}f_t(s)\,ds=\int_0^\infty e^{-\lambda s}\frac{\partial^q}{\partial t^q}f_t(s)\,ds
\]
(see the proof of Lemma 3.1 in \cite{Deb}). Then integration by parts (using \eqref{eq2.2}, \eqref{eq2.4} and \eqref{eq2.5} to see that the boundary terms are 0) gives
\begin{align*}
\int_0^\infty e^{-\lambda s}\frac{\partial}{\partial s}f_t(s)\,ds&=\lambda\int_0^\infty e^{-\lambda s}f_t(s)\,ds\\
&=\lambda\exp(-t\lambda^{\alpha/2})=\lambda\exp(-t\lambda^{1/m})\\
&=(-1)^m\frac{\partial^m}{\partial t^m}\exp(-t\lambda^{1/m})=(-1)^m\frac{\partial^m}{\partial t^m}\exp(-t\lambda^{\alpha/2})\\
&=(-1)^m\frac{\partial^m}{\partial t^m}\int_0^\infty e^{-\lambda s}f_t(s)\,ds\\
&=\int_0^\infty e^{-\lambda s}(-1)^m\frac{\partial^m}{\partial t^m}f_t(s)\,ds.
\end{align*}
\end{proof}

The next result is taken from DeBlassie \cite[Propositions 4.7 and 4.8]{Deb}.
\begin{lem}\label{lem2.2}
Let $D\subseteq\mathbb{R}^d$ be a bounded Lipschitz domain and suppose $\alpha=k/m\in(0,1)$ is rational. Then for $x\in\mathbb{R}^d$ and $t>0$,
\[
\frac{\partial u_n}{\partial t}(t,x)=-\lambda_nu_n(t,x)+P_tr_n(x),
\]
where the integrable function $r_n$ is given by
\[
r_n(x)=
\left\{
\begin{array}{ll}
\displaystyle\int_D\dfrac{c_{d,\alpha}\,\varphi_n(y)}{|x-y|^{d+\alpha}}\,dy,&\quad x\in\text{int($D^c$)}\\
&\\
0,&\quad x\in\overline{D},
\end{array}
\right.
\]
and
\[
c_{d,\alpha}=2^\alpha\pi^{-1-d/2}\Gamma\left(\frac{d+\alpha}{2}\right)\Gamma\left(1+\frac{\alpha}{ 2}\right)\sin\frac{\pi\alpha}{2}.\qquad\qquad\qquad\square
\]
\end{lem}

Using the methods from \cite{Deb}, the argument used to prove Theorem 4.7 in Ba\~nuelos and Kulczycki \cite{BanKul} can be modified to yield the following Lemma.

\begin{lem}\label{lem2.3}
Suppose $D\subseteq\mathbb{R}^d$ is a bounded Lipschitz domain and $\alpha=k/m\in(0,1]$ is rational.  Then for $x\in D$,
\[
\Delta_xu_1(t,x)\to\Delta_x\varphi_1(x),\quad\text{as $t\to 0^+.$}\qquad\qquad\qquad\square
\]
\end{lem}

The final result we will need below is an easy extension of Lemma 2.1 in DeBlassie \cite{Deb}.

\begin{lem}\label{lem2.4}
Suppose $\varphi$ is integrable on $\mathbb{R}^d$. Then for any integer $q\geq 0$,
\[
\frac{\partial^q}{\partial t^q}P_t\varphi(x)=\int_{\mathbb{R}^d}\int_0^\infty\varphi(y)g(s,x,y)\frac{\partial^q}{\partial t^q}f_t(s)\,ds\,dy.
\]
In addition, if $\varphi$ is bounded with compact support, then for any multi-index $\gamma=(\gamma_1,\dots,\gamma_d)$,
\[
D^\gamma_xP_t\varphi(x)=\int_{\mathbb{R}^d}\int_0^\infty\varphi(y)D^\gamma_xg(s,x,y)f_t(s)\,ds\,dy,\qquad D^\gamma_x=\frac{\partial^{\gamma_1}}{\partial x^{\gamma_1}}\cdots\frac{\partial^{\gamma_d}}{\partial x^{\gamma_d}}.\qquad\square
\]
\end{lem}

\bigskip

\section{Proof of Theorem \ref{thm1.1}}\label{sec3}

We assume the hypotheses of Theorem \ref{thm1.1} throughout this section. 

By Lemmas \ref{lem2.1} and \ref{lem2.4}, integration by parts (where we also use \eqref{eq2.2}--\eqref{eq2.5} to see that the boundary terms are 0) yields

\begin{align*}
\Delta_x u_1(t,x)&=\Delta_x P_t\varphi_1(x)=\int_{\mathbb{R}^d}\int_0^\infty\varphi_1(y)\Delta_xg(s,x,y)f_t(s)\,ds\,dy\\
&=\int_{\mathbb{R}^d}\int_0^\infty\varphi_1(y)\left[\frac{\partial}{\partial s}g(s,x,y)\right]f_t(s)\,ds\,dy\\
&=-\int_{\mathbb{R}^d}\int_0^\infty\varphi_1(y)g(s,x,y)\frac{\partial}{\partial s}f_t(s)\,ds\,dy\\
&=(-1)^{m+1}\int_{\mathbb{R}^d}\int_0^\infty\varphi_1(y)g(s,x,y)\frac{\partial^m}{\partial t^m}f_t(s)\,ds\,dy\\
&=(-1)^{m+1}\frac{\partial^m}{\partial t^m}P_t\varphi_1(x)\\
&=(-1)^{m+1}\frac{\partial^m}{\partial t^m}u_1(t,x).
\end{align*}

By Lemma \ref{lem2.4} $P_tr_1$ is infinitely differentiable in $t$, so by Lemma \ref{lem2.2} we have
\[
\frac{\partial^m u_1}{\partial t^m}(t,x)=(-1)^m\left(\lambda_1^{m}u_1(t,x)+\sum_{q=0}^{m-1}(-1)^{q+1}\,\lambda_1^{m-1-q}\,\frac{\partial^q}{\partial t^q}P_tr_1(x)\right). 
\]
Thus, 
\[
\Delta_x u_1(t,x)=-\lambda_1^{m}u_1(t,x)+\sum_{q=0}^{m-1}(-1)^{q}\,\lambda_1^{m-1-q}\,\frac{\partial^q}{\partial t^q}P_tr_1(x).
\]
By Lemma \ref{lem2.3}, upon letting $t\to 0^+$, we get
\[
\Delta_x\varphi_1(x)=-\lambda_1^{m}\varphi_1(x)+\lim_{t\to 0^+}\sum_{q=0}^{m-1}(-1)^{q}\,\lambda_1^{m-1-q}\,\frac{\partial^q}{\partial t^q}P_tr_1(x).
\]
Thus the conclusion of Theorem \ref{thm1.1} will follow once we show 
\begin{equation}\label{eq3.1}
\lim_{t\to 0^+}(-1)^{q}\,\frac{\partial^q}{\partial t^q}P_tr_1(x)\leq 0,\quad q=0,\dots,m-1.
\end{equation}

In order to prove \eqref{eq3.1} we need the following technical lemmas, whose proofs we defer to the next section.

\begin{lem}\label{lem3.1}
Let $x\in D$ and set $c=d(x,D^c)^2/4$ (which is positive). Then given a positive integer $q$ and $M>1$, for some $c_1(q,M)>0$ independent of $t$ we have 
\[
\sup_{y\in D^c}\left|\int_0^{Mt^{2/\alpha}}g(s,x,y)\,\frac{\partial^q}{\partial t^q}f_t(s)\,ds\right|\leq c_1(q,M)\,t^{-2/\alpha-q}\int_0^{Mt^{2/\alpha}} s^{-d/2}
\,e^{-c/s}\,ds\to 0,
\]
as $t\to 0^+$.
\end{lem}

\begin{lem}\label{lem3.2}
Let $x\in D$. Then given a positive integer $q$ and $M>1$, there is $\beta(q,M)>0$ such that
\[
\sup_{t\leq 1}\sup_{y\in D^c}\left|t^{-q-2/\alpha}\sum_{j=0}^qa_j(q)\int_{Mt^{2/\alpha}}^\infty s^j\,g(s,x,y)\frac{\partial^j}{\partial s^j}\left(\sum_{k=q+1}^\infty\gamma_k(t^{-2/\alpha}s)^{-k\alpha/2-1}\right)\,ds\right|\leq \beta(q,M),
\]
where $\beta(q,M)$ converges to $0$ as $M\to\infty$.
\end{lem}

\begin{lem}\label{lem3.3}
For $x\in D$ and $q$ a nonnegative integer, 
\[
\sup_{t\leq 1}\sup_{y\in D^c}\left|t^{-q-2/\alpha}\sum_{j=0}^qa_j(q)\int_{Mt^{2/\alpha}}^\infty s^j\,g(s,x,y)\frac{\partial^j}{\partial s^j}\left(\sum_{k=1}^q\gamma_k(t^{-2/\alpha}s)^{-k\alpha/2-1}\right)\,ds\right|<\infty
\]
and the inside of the absolute value converges to
\[
(q!\,)\,\gamma_q\int_0^\infty s^{-q\alpha/2-1}\,g(s,x,y)\,ds, 
\]
as $t\to 0^+$.
\end{lem}

We now show how these lemmas imply \eqref{eq3.1}. First, note that by Lemma \ref{lem2.2}, $r_1$ is integrable and vanishes on $D$, so by Lemma \ref{lem2.4} we have for $x\in D$,
\[
\frac{\partial^q}{\partial t^q}\,P_tr_1(x)=\int_{D^c}\left(\int_0^\infty g(s,x,y)\,\frac{\partial^q}{\partial t^q}f_t(s)\,ds\right)\,r_1(y)\,dy.
\]
Next, given a positive integer $q$ and $M>1$,
\begin{equation}\label{eq3.2}
\int_0^\infty g(s,x,y)\,\frac{\partial^q}{\partial t^q}f_t(s)\,ds=\int_0^{Mt^{2/\alpha}} g(s,x,y)\,\frac{\partial^q}{\partial t^q}f_t(s)\,ds
+\int_{Mt^{2/\alpha}}^\infty g(s,x,y)\,\frac{\partial^q}{\partial t^q}f_t(s)\,ds,
\end{equation}
and by \eqref{eq2.3}-\eqref{eq2.4},
\begin{align}\label{eq3.3}
\int_{Mt^{2/\alpha}}^\infty g(s,x,y)\,\frac{\partial^q}{\partial t^q}f_t(s)\,ds
&=\sum_{j=0}^qa_j(q)\,t^{-q-2/\alpha}
\int_{Mt^{2/\alpha}}^\infty g(s,x,y)\,(t^{-2/\alpha}s)^jf_1^{(j)}(t^{-2/\alpha}s)\,ds\notag\\
&=t^{-q-2/\alpha}\sum_{j=0}^qa_j(q)
\int_{Mt^{2/\alpha}}^\infty s^j\,g(s,x,y)\,\frac{\partial^j}{\partial s^j}f_1(t^{-2/\alpha}s)\,ds\notag\\
&=t^{-q-2/\alpha}\sum_{j=0}^qa_j(q)
\int_{Mt^{2/\alpha}}^\infty s^j\,g(s,x,y)\,\frac{\partial^j}{\partial s^j}\left(\sum_{k=1}^\infty\gamma_k(t^{-2/\alpha}s)^{-k\alpha/2-1}\right)\,ds\notag\\
&=t^{-q-2/\alpha}\sum_{j=0}^qa_j(q)
\int_{Mt^{2/\alpha}}^\infty s^j\,g(s,x,y)\,\frac{\partial^j}{\partial s^j}\left(\sum_{k=1}^q+\sum_{k=q+1}^\infty\right)\times\notag\\ &\left(\gamma_k\,(t^{-2/\alpha}s)^{-k\alpha/2-1}\right)\,ds, 
\end{align}
where in the second equality we used the fact that $$\frac{\partial^j}{\partial s^j}f_1(t^{-2/\alpha}s)= t^{(-2/\alpha) j}f_1^{(j)}(t^{-2/\alpha}s).$$
Hence by Lemmas \ref{lem3.1}-\ref{lem3.3}, we see that for fixed $x\in D$, the left hand side of \eqref{eq3.2} is bounded for $t\leq 1$ and $y\in D^c$. Since $r_1$ is integrable on $\mathbb{R}^d$ and vanishes on $D$, we can apply dominated convergence to get for $M>1$,
\begin{align*}
\lim_{t\to 0^+}\frac{\partial^q}{\partial t^q}P_tr_1(x)
&=\lim_{t\to 0^+}\Bigg[\int_{D^c}\left(\int_0^{Mt^{2/\alpha}} g(s,x,y)\,\frac{\partial^q}{\partial t^q}f_t(s)\,ds\right)\,r_1(y)\,dy\\
&\hspace{1in}+\int_{D^c}\left(\int_{Mt^{2/\alpha}}^\infty g(s,x,y)\,\frac{\partial^q}{\partial t^q}f_t(s)\,ds\right)\,r_1(y)\,dy\Bigg]\\
&=\left[ 0+(q!\,)\,\gamma_q\int_{D^c}\left(\int_0^\infty s^{-q\alpha/2-1}\,g(s,x,y)\,ds\right)\,r_1(y)\,dy+O(\beta(q,m))\right],
\end{align*}
by \eqref{eq3.3} and Lemmas \ref{lem3.1}-\ref{lem3.3}. Then let $M\to\infty$ to obtain 
\[
\lim_{t\to 0^+}\frac{\partial^q}{\partial t^q}P_tr_1(x)=(q!)\,\gamma_q\int_{D^c}\left(\int_0^\infty s^{-q\alpha/2-1}\,g(s,x,y)\,ds\right)\,r_1(y)\,dy.
\]

To finish the proof of \eqref{eq3.1}, since $r_1\geq 0$, it remains to show that
\[
(-1)^q\,\gamma_q\leq 0,\quad q=0,\dots,m-1.
\]
To see this, observe that for some $R_q>0$, 
\[
\gamma_q=R_q\,(-1)^{q+1}\,\sin\left(\frac{\pi q\alpha}{2}\right).
\]
Since $0\leq q\leq m-1$ and $\alpha=2/m$,
\[
\frac{\pi q\alpha}{2}\in\left[0,\pi(1-1/m)\right],
\] 
and this implies that
\[
\sin\left(\frac{\pi q\alpha}{2}\right)\geq 0.
\]
It follows that for some $\rho_q\geq 0$, 
\[
\gamma_q=\rho_q\,(-1)^{q+1},
\]
and this yields that $(-1)^q\gamma_q\leq 0$, as desired.\hfill$\square$

\section{Proof of Lemmas \ref{lem3.1}-\ref{lem3.3}}\label{sec4}

\bigskip
\n{\bf Proof of Lemma \ref{lem3.1}}. For $M>1$ and $s\leq Mt^{2/\alpha}$, by \eqref{eq2.3} and \eqref{eq2.5},
\begin{align*}
\left|\frac{\partial^q}{\partial t^q}f_t(s)\right|
&=\left| \sum_{j=0}^qa_j(q)\,t^{-2/\alpha-q}\,(t^{-2/\alpha}s)^jf_1^{(j)}(t^{-2/\alpha}s)\right|\\
&\leq c_1(q,M)\,t^{-2/\alpha-q},
\end{align*}
where $c_1(q,M)>0$ is independent of $t$. Thus for fixed $x\in D$,
\begin{align*}
\sup_{y\in D^c}\left|\int_0^{Mt^{2/\alpha}}g(s,x,y)\,\frac{\partial^q}{\partial t^q}f_t(s)\,ds\right|
&\leq c_2\,\int_0^{Mt^{2/\alpha}}s^{-d/2}\,e^{-c/s}\left|\frac{\partial^q}{\partial t^q}f_t(s)\right|\,ds\\
&\leq c_3\,t^{-2/\alpha-q}\int_0^{Mt^{2/\alpha}} s^{-d/2}
\,e^{-c/s}\,ds\\
&\to 0, 
\end{align*}
as $t\to 0^+$.\hfill$\square$

\bigskip
\n{\bf Proof of Lemma \ref{lem3.2}}. We have for $M>1$,
\begin{align*}
&\left|t^{-q-2/\alpha}\int_{Mt^{2/\alpha}}^\infty s^j\,g(s,x,y)\frac{\partial^j}{\partial s^j}\left(\sum_{k=q+1}^\infty\gamma_k(t^{-2/\alpha}s)^{-k\alpha/2-1}\right)\,ds\right|\\
&\qquad=\left|t^{-q-2/\alpha}\int_{Mt^{2/\alpha}}^\infty s^j\,g(s,x,y)\left(\sum_{k=q+1}^\infty\gamma_k\,(-1)^j\,\frac{\Gamma(k\alpha/2+1+j)}{\Gamma(k\alpha/2+1)}\,t^{-2j/\alpha}\,(t^{-2/\alpha}s)^{-k\alpha/2-1-j}\right)\,ds\right|\\
&\qquad\leq t^{-q-2/\alpha}\sum_{k=q+1}^\infty\left|\gamma_k\right|\,\frac{\Gamma(k\alpha/2+1+j)}{\Gamma(k\alpha/2+1)}\,\int_{Mt^{2/\alpha}}^\infty(t^{-2/\alpha}s)^{-k\alpha/2-1}\,g(s,x,y)\,ds\\
&\qquad=t^{-q-2/\alpha}\sum_{\ell=0}^\infty\left|\gamma_{\ell+q+1}\right|\,\frac{\Gamma((\ell+q+1)\alpha/2+1+j)}{\Gamma((\ell+q+1)\alpha/2+1)}\,\int_{Mt^{2/\alpha}}^\infty(t^{-2/\alpha}s)^{-(\ell+q+1)\alpha/2-1}\,g(s,x,y)\,ds\\
&\qquad=\sum_{\ell=0}^\infty\left|\gamma_{\ell+q+1}\right|\,\frac{\Gamma((\ell+q+1)\alpha/2+1+j)}{\Gamma((\ell+q+1)\alpha/2+1)}\,\int_{Mt^{2/\alpha}}^\infty(t^{-2/\alpha}s)^{-(\ell+1)\alpha/2}\,s^{-q\alpha/2-1}\,g(s,x,y)\,ds\\
&\qquad\leq \sum_{\ell=0}^\infty\left|\gamma_{\ell+q+1}\right|\,\frac{\Gamma((\ell+q+1)\alpha/2+1+j)}{\Gamma((\ell+q+1)\alpha/2+1)}\,M^{-(\ell+1)\alpha/2}\,\int_{Mt^{2/\alpha}}^\infty s^{-q\alpha/2-1}\,g(s,x,y)\,ds.
\end{align*}
Therefore for fixed $x\in D$ and $M>1$, with $c=d(x,D^c)^2/4$,
\begin{align*}
\sup_{y\in D^c}&\left|t^{-q-2/\alpha}\sum_{j=0}^qa_j(q)\int_{Mt^{2/\alpha}}^\infty s^j\,g(s,x,y)\frac{\partial^j}{\partial s^j}\left(\sum_{k=q+1}^\infty\gamma_k(t^{-2/\alpha}s)^{-k\alpha/2-1}\right)\,ds\right|\\
&\qquad\leq \left(\sum_{j=0}^q\left|a_j(q)\right|\sum_{\ell=0}^\infty\left|\gamma_{\ell+q+1}\right|\,\frac{\Gamma((\ell+q+1)\alpha/2+1+j)}{\Gamma((\ell+q+1)\alpha/2+1)}\,M^{-(\ell+1)\alpha/2}\right)\times\\
&(4\pi)^{-d/2}\int_0^\infty s^{-q\alpha/2-1-d/2}\,e^{-c/s}\,ds=\beta{(q,M)},
\end{align*}
where $\beta{(q,M)}\to 0$ as $M\to\infty$ is independent of $t$.\hfill$\square$

In order to prove Lemma \ref{lem3.3}, we need the following result.
\begin{lem}\label{lem4.1}
The following identity holds:
\[
\sum_{j=0}^q a_j(q)\,(-1)^j\,\frac{\Gamma(k\alpha/2+1+j)}{\Gamma(k\alpha/2+1)}
=\left\{\begin{array}{ll}0,&\qquad k<q\\q!,&\qquad k=q.\end{array}\right.
\]
\end{lem}
\begin{proof}
By \eqref{eq2.2}, $f_t(s)=t^{-2/\alpha}\,f_1(t^{-2/\alpha}s)$, hence the formula \eqref{eq2.3} is a special case of the formula
\begin{equation}\label{eq4.1}
\frac{\partial^q}{\partial t^q}\,t^{-2/\alpha}\,h(t^{-2/\alpha}s)=\sum_{j=0}^qa_j(q)\,t^{-2/\alpha-q}(t^{-2/\alpha}s)^jh^{(j)}(t^{-2/\alpha}s).
\end{equation}
Taking $h(x)=x^{-k\alpha/2-1}$ we have
\[
h^{(j)}(x)=(-1)^j\frac{\Gamma(k\alpha/2+1+j)}{\Gamma(k\alpha/2+1)}\,x^{-k\alpha/2-1-j}.
\]
Then with $s=1$, the right hand side of \eqref{eq4.1} becomes
\begin{align*}
\sum_{j=0}^qa_j(q)t^{-2/\alpha-q}&(t^{-2/\alpha})^jh^{(j)}(t^{-2/\alpha})\\
&=\sum_{j=0}^q(-1)^j\,a_j(q)\,t^{-2/\alpha-q}\,(t^{-2/\alpha})^j\,\frac{\Gamma(k\alpha/2+1+j)}{\Gamma(k\alpha/2+1)}\,(t^{-2/\alpha})^{-k\alpha/2-1-j}\\
&=t^{k-q}\sum_{j=0}^q(-1)^j\,a_j(q)\,\frac{\Gamma(k\alpha/2+1+j)}{\Gamma(k\alpha/2+1)}.
\end{align*}
On the other hand,
\[
t^{-2/\alpha}\,h(t^{-2/\alpha})=t^{-2/\alpha}\,(t^{-2/\alpha})^{-k\alpha/2-1}=t^k,
\]
and so \eqref{eq4.1} becomes
\[
\frac{\partial^q}{\partial t^q}\,t^k=t^{k-q}\sum_{j=0}^q(-1)^j\,a_j(q)\,\frac{\Gamma(k\alpha/2+1+j)}{\Gamma(k\alpha/2+1)}.
\]
The desired conclusion follows from this.
\end{proof}

\bigskip
\n{\bf Proof of Lemma \ref{lem3.3}}. We have
\begin{align*}
t^{-q-2/\alpha}\sum_{j=0}^q&a_j(q)\int_{Mt^{2/\alpha}}^\infty s^j\,g(s,x,y)\frac{\partial^j}{\partial s^j}\left(\sum_{k=1}^q\gamma_k(t^{-2/\alpha}s)^{-k\alpha/2-1}\right)\,ds\\
&=t^{-q-2/\alpha}\sum_{j=0}^qa_j(q)\,\sum_{k=1}^q\gamma_k\int_{Mt^{2/\alpha}}^\infty s^j\,g(s,x,y)\,t^{k+2/\alpha}\,(-1)^j\,\frac{\Gamma(k\alpha/2+1+j)}{\Gamma(k\alpha/2+1)}s^{-k\alpha/2-1-j}\,ds\\
&=\sum_{j=0}^qa_j(q)\,(-1)^j\,\left(\int_{Mt^{2/\alpha}}^\infty s^{-k\alpha/2-1}\,g(s,x,y)\,ds\right)
\sum_{k=1}^q\gamma_k\,\frac{\Gamma(k\alpha/2+1+j)}{\Gamma(k\alpha/2+1)}\,t^{k-q}\\
&=\sum_{k=1}^q\gamma_k\,t^{k-q}\,\left(\int_{Mt^{2/\alpha}}^\infty s^{-k\alpha/2-1}\,g(s,x,y)\,ds\right)
\sum_{j=0}^qa_j(q)\,(-1)^j\,\frac{\Gamma(k\alpha/2+1+j)}{\Gamma(k\alpha/2+1)}\\
&=(q!\,)\,\gamma_q\int_{Mt^{2/\alpha}}^\infty s^{-q\alpha/2-1}\,g(s,x,y)\,ds,
\end{align*}
by Lemma \ref{lem4.1}. In particular, for fixed $x\in D$ and $c=d(x,D^c)^2/4$,
\begin{align*}
\sup_{t\leq 1}\sup_{y\in D^c}&\left|t^{-q-2/\alpha}\sum_{j=0}^qa_j(q)\int_{Mt^{2/\alpha}}^\infty s^j\,g(s,x,y)\frac{\partial^j}{\partial s^j}\left(\sum_{k=1}^q\gamma_k(t^{-2/\alpha}s)^{-k\alpha/2-1}\right)\,ds\right|\\
&\leq (q!\,)\,|\gamma_q|(4\pi)^{-d/2}\int_{0}^\infty s^{-q\alpha/2-1-d/2}\,e^{-c/s}\,ds<\infty.
\end{align*}
Moreover, as $t\to 0^+$,
\begin{align*}
t^{-q-2/\alpha}\sum_{j=0}^qa_j(q)&\int_{Mt^{2/\alpha}}^\infty s^j\,g(s,x,y)\frac{\partial^j}{\partial s^j}\left(\sum_{k=1}^q\gamma_k(t^{-2/\alpha}s)^{-k\alpha/2-1}\right)\,ds\\
&\to(q!\,)\,\gamma_q\int_{0}^\infty s^{-q\alpha/2-1}\,g(s,x,y)\,ds,
\end{align*}
as desired.\hfill$\square$

\end{document}